\documentclass[12pt]{article}
\usepackage{geometry}

\usepackage[utf8]{inputenc}
\usepackage{amsmath}
\usepackage{amsfonts}
\usepackage{amssymb}
\usepackage[normalem]{ulem}
\usepackage{booktabs}
\usepackage{amsthm}
\usepackage{authblk}
\usepackage{hyperref}

\usepackage{float}

\usepackage{color}

\DeclareMathOperator*{\argmax}{arg\,max}
\DeclareMathOperator*{\argmin}{arg\,min}
\newtheorem{definition}{Definition}
\newtheorem{proposition}{Proposition}
\newtheorem{corollary}{Corollary}
\newtheorem{theorem}{Theorem}
\newtheorem{conjecture}{Conjecture}

\usepackage{tikz}

\usepackage[style=numeric,backend=bibtex]{biblatex}

\usepackage{todonotes}

\usepackage{algpseudocode}
\usepackage{algorithmicx}
\usepackage{algorithm}

\title{Characterization of Logarithmic Fekete Critical Configurations of at Most Six Points in All Dimensions}

\author[1]{Diego Armentano}
\author[2]{Leandro Bentancur}
\author[1,3]{Federico Carrasco}
\author[3]{\\Marcelo Fiori}
\author[4]{Matías Valdés\thanks{matias.valdes@noreste.udelar.edu.uy}}
\author[2]{Mauricio Velasco}

\affil[1]{\footnotesize Facultad de Ciencias Económicas y Administración, Universidad de la República, Montevideo, Uruguay}
\affil[2]{Facultad de Ciencias, Universidad de la República, Montevideo, Uruguay}
\affil[3]{Facultad de Ingeniería, Universidad de la República, Montevideo, Uruguay}

\affil[4]{Centro Universitario Regional Noreste, Universidad de la República, Tacuaremb\'o, Uruguay}

\date{}

\providecommand{\keywords}[1]{\textit{Keywords:} #1} 

\begin{document}

\maketitle

\begin{abstract}
	We consider the logarithmic Fekete problem, which consists of placing a fixed number of points on the unit sphere in $\mathbb{R}^d$, in such a way that the product of all pairs of mutual Euclidean distances is maximized or, equivalently, so that their logarithmic energy is minimized.
	Using tools from Computational Algebraic Geometry, we find and classify all critical configurations for this problem when considering at most six points in every dimension $d$.
    In particular, our approach gives new proofs of several key results appearing in the literature, with the benefit of using a unified approach.
    {Furthermore, for seven points in $S^2$, we characterize the global minimizer among critical configurations having at least one pair of antipodal points, and {give numerical evidence to support the} conjecture that this configuration is also the unrestricted global minimizer.}
\end{abstract}



\section{Introduction} \label{sec:intro}
	
	Consider the problem of placing $n$ different points on the unit sphere $S^{d-1} \subset \mathbb{R}^d$, in such a way that the product of their mutual Euclidean distances is maximized:
	\begin{equation} \label{eq:fekete_prod}
		\argmax_{w=(w_1,\dots, w_n) \in (S^{d-1})^n } \prod_{i=1}^{n} \prod_{j=i+1}^{n} \| w_i - w_j \|^2 .
	\end{equation}
    Taking logarithm in the objective function, we obtain an equivalent problem, known as the logarithmic Fekete problem:
	\begin{equation} \label{eq:fekete_log}
		\argmin_{ w \in (S^{d-1})^n } - \sum_{i=1}^{n} \sum_{j=i+1}^{n} \log \left( \|w_i - w_j \|^2 \right) .
	\end{equation}

    This is considered a highly non-trivial optimization problem, with exact solutions known only for a few values of $n$.
    Indeed, Smale’s 7th problem, listed among the key open problems for the 21st century by Steve Smale, asks whether it is possible to find $n$ points on the sphere $S^2$ in polynomial time in $n$, so that its logarithmic energy differs from the optimal value by at most $c \ln n$, for a universal constant $c$ \cite{smale_problems}. A key difficulty of this problem is that the optimal value remains unknown to logarithmic precision \cite{betermin2018renormalized,carlos_fatima}.
    
    There are several lines of research related to this problem. Among them, we highlight two. 
    One focuses on constructing random or deterministic configurations with ``good'' energy values for an arbitrary number of points, usually large or asymptotic \cite{ABS,alishahi2015spherical,beltran2018diamond}. The other aims to characterize optimal configurations for specific values of $n$, typically small ones. Our work falls within this latter setting.
    
    While the solution to this problem is known only for a few values of $n$, namely, $n \leq 6$ and $n = 12$, even less is understood about the critical configurations.
    For example, even for some values of $n$ where the solution is known, {the complete set of critical configurations is unknown.}
    
	In this work, {using a unified methodology,} we find and classify all the critical configurations of the logarithmic Fekete problem, for the case of at most six points, and for spheres of all possible dimensions. {In this setting, we recover the previously reported solutions, and we show that there are no spurious local minima in $S^2$.}
    
	Our strategy for determining all critical configurations of Problem~\eqref{eq:fekete_log} proceeds as follows.
	First we define a system of polynomial equations, associated to the critical configurations.
	Then we count the number of complex solutions of the polynomial system, by computing a Gr{\"o}bner basis for its ideal. We call this quantity the ``expected number of solutions''. Observe that regardless of whether the ideal is radical, this is an upper bound on the number of real solutions of the problem.
	Simultaneously, we find as many solutions of the polynomial system as we can, and compare the number of found solutions with the upper bound. The proof of the Theorem is achieved because we are able to find enough solutions so as to exactly match the upper bound, guaranteeing that our list of critical configurations is exhaustive.
	The construction of solutions is done using different approaches. In the case of four and five points, we are able to find all solutions by considering natural symmetric configurations.  For six points, imagination is not enough, and we use Gr{\"o}bner bases to find additional solutions. 
    The only problem with the above strategy is that the locus of critical points in $(S^{d-1})^n$ does not form a zero-dimensional variety because it is invariant under the action of the orthogonal group. In order to recover finiteness, we reformulate the problem modulo the orthogonal group action. Doing so recovers finiteness and leads to a simpler formulation, which allows us to carry out the desired classification.
    The source code for reproducing our results is available at: \url{www.github.com/matiasvd}.

\section{Related work} \label{sec:related_work}

	For $S^1$, the solution of Problem \eqref{eq:fekete_log} is known for any number of points, and is given by $n$ equidistributed points \cite[Theorem 2.3.3]{borodachov2019}.
    We denote this configuration as $n$-gon or ``Equator''.
	For $S^2$, the solutions are known only for up to six points \cite{andreev1996, kolushov1997, dragnev2002}, and for twelve points \cite{andreev1996}.
    {For $S^3 \subset \mathbb{R}^4$ and six points, the solution is given by two equilateral triangles, inscribed in two great circles $S^1$, orthogonal to each other \cite[Theorem 1.9]{Dragnev2016}. We denote this configuration as $3_{S^1} \times 3_{S^1}$.}
	Table~\ref{tab:known_solutions} lists the known solutions for up to six points, for all possible dimensions of the sphere.
	We use the term $(n-1)$-simplex for the configuration that has $n$ points on the sphere {$S^{d-1}$}, all at the same distance {from each other. The $(n-1)$-simplex is the optimal configuration for $n$ points on $S^{d-1}$, whenever $d+1 \geq n \geq 3$ \cite{kolushov1997}.} Note that the 3-simplex is the regular Tetrahedron.
    
	\begin{table}[H]
        \centering
		\caption{Known optimal configurations of the logarithmic Fekete problem for at most six points.}
		\label{tab:known_solutions}
		\begin{tabular}{ccccc}
			\toprule
			$n$ & $S^1$ & $S^2$ & $S^3$ & $S^4$ \\
			\midrule
			3 & 2-simplex & N/A & N/A & N/A \\
			4 & 4-gon & Tetrahedron \cite{kolushov1997} & N/A & N/A \\
			5 & 5-gon & {Bipyramid} \cite{dragnev2002} & 4-simplex \cite{kolushov1997} & N/A \\
			6 & 6-gon & Octahedron \cite{kolushov1997} & {$3_{S^1} \times 3_{S^1}$ \cite{Dragnev2016} } & 5-simplex \cite{kolushov1997} \\
			\bottomrule
		\end{tabular}
	\end{table}

    Regarding critical configurations of Problem \eqref{eq:fekete_log}, little is known about them, even for the cases where the solutions are known.
    One thing that is known is that the problem has no local maxima \cite[Corollary 1.3]{beltran2013}, and that critical configurations always have center of mass zero \cite[Proposition 2]{dragnev2002}.
    It is also known that the global minima is not always unique. In particular, if $q=p^l$, with $p>2$ prime and $l \geq 1$, taking $n=(q+1)(q^3+1)$ points on $S^{d-1}$, $d=q \frac{q^3+1}{q+1}$, there are $\lfloor (l-1)/2 \rfloor$ essentially different global minima \cite[Section 1]{ballinger2009}.
    {On the other hand, in \cite{Dragnev2016, Dragnev2023}, the authors study the critical configurations for $d+2$ points on $S^{d-1}$, and give a method to classify all critical configurations, and to characterize the non-degenerate ones. Moreover, for six points in $S^3$, they find the global minimum and give the first known case of a spurious local minima.}
    {Also, the only saddle points that are explicitly mentioned in the literature are the Equator \cite[Section 2]{shub1993}, and two configurations for six points in $S^3$ \cite{Dragnev2016, Dragnev2023}. }
    Another natural question is whether every saddle may be classified using the Hessian. This is particularly important for numerical optimization \cite{ge2015, jin2017}. For 7 points on $S^2$, there is a critical configuration that the Hessian does not classify \cite[Remark 6.3]{constantineau2023}, although it is believed that it is a global minimum \cite{beltran2020}[Conjecture 11.1].
    Finally, based on numerical experiments, it is conjectured that the number of spurious local minima {in $S^2$} increases ``dramatically'' with the number of points \cite[Section 4]{rakhmanov1995}, although no spurious local minima are known in this dimension.\\

\noindent \textbf{Contributions}\\

    {We introduce an approach that transforms the real optimization problem of Fekete points into a novel problem amenable to the methods of Computational Algebraic Geometry (over the complex numbers). In principle, this methodology allows us to find and classify all critical configurations of $n$ points in $S^{d-1}$, without assuming any relation between $n$ and $d$, modulo sufficient computational resources. We apply this method successfully for up to six points in all dimensions, recovering previously known results in a unified framework.}
    
    {Although the optimal configuration for six points in $S^2$ was already reported, our method obtains the first exhaustive list of all critical configurations and their classification.}    
    {For six points in $S^3$, since our method obtains all critical configurations, including the degenerate ones, this complements the results of \cite{Dragnev2016, Dragnev2023}, giving an exhaustive list of all critical configurations in this case.}

    {Finally, for seven points in $S^2$, we prove that if a solution to the Fekete problem has a dipole, then it is given by {configuration 1:5:1, which is} the solution conjectured via numerical experiments.}
    {If 1:5:1 is indeed the global minimum, our theorem reduces the problem of proving that it is optimal to verifying that some minimizer of the Fekete problem must have a dipole. We also provide numerical evidence in favor of the optimality of 1:5:1.}
    
\section{Formulation and computational algebraic tools}

\subsection{System of critical configurations} \label{sec:critical_confs}

	The Lagrangian of Problem \eqref{eq:fekete_log} is:
	$$ L(w, \lambda) := - \sum_{i=1}^{n} \sum_{j=i+1}^{n} \log \left( \|w_i - w_j \|^2 \right) + \sum_{i=1}^{n} \lambda_i \left( \| w_i \|^2 - 1 \right) .$$
	A critical configuration is a sequence $(w_1,\ldots,w_n)$ on the product of spheres, for which there exist multipliers $( \lambda_1,\ldots,\lambda_n )$, such that the Lagrangian has null gradient with respect to $w$:
	$$ \frac{\partial L(w,\lambda)}{\partial w_k} = - \sum_{j=1, j \neq k}^{n} \frac{ 2 \left( w_k - w_j \right) }{ \| w_k - w_j \|^2 } + 2 \lambda_k w_k = \vec{0} , \quad \forall \ k=1,\ldots,n .$$
	Taking dot product with $w_k$, and using that $\|w_k - w_j \|^2 = 2 - 2 w_k^T w_j$ on the sphere, we obtain the values of the multipliers, which happen to be all equal and constant:
    $$ \lambda_k = \frac{n-1}{2} , \quad \forall \ k=1,\ldots,n .$$
	Thus, the critical configurations are those points on the sphere
	\begin{equation} \label{eq:shpere}
		\|w_k\|^2 = 1, \quad \forall \ k=1,\ldots,n ,
	\end{equation}
	which also satisfy the null gradient condition:
	\begin{equation} \label{eq:lagrange_gradient_null}
		\frac{n-1}{2} w_k - \sum_{j=1, j \neq k}^{n} \frac{ w_k - w_j }{ \| w_k - w_j \|^2 } = \vec{0}, \quad \forall \ k=1,\ldots,n .\\
	\end{equation}
    
    An interesting property of the critical configurations of Problem \eqref{eq:fekete_log}, is that their center of mass must be null:
	\begin{equation} \label{eq:center_mass_null}
		\sum_{k=1}^{n} w_k = \vec{0} .
	\end{equation}    
	This is obtained by summing the equations \eqref{eq:lagrange_gradient_null} and noting that:
	$$ \sum_{k=1}^{n} \sum_{j=1, j \neq k}^{n} \frac{ w_k - w_j }{ \| w_k - w_j \|^2 } = \vec{0} ,$$
    {because for every term $w_k-w_j$, the summation includes a corresponding $w_j-w_k$ term.}
    
    Searching for critical configurations of the Fekete problem means finding solutions to the system of polynomial equations defined by \eqref{eq:shpere} and \eqref{eq:lagrange_gradient_null}, to which we will add \eqref{eq:center_mass_null}, which follows from the previous ones and does not modify the set of solutions.
    
\subsection{Removing orthogonal symmetry} \label{sec:removing_orthogonal_symmetry}

    An important property of the system of equations defined in the preceding section is that the system is invariant under isometries of the sphere.
    More precisely, if $Q \in \mathbb{R}^{d \times d}$ is an orthogonal matrix, and $(w_1, \ldots, w_n)$ is a critical configuration, then $(Q w_1, \ldots, Q w_n)$ is also a critical configuration.
    Thus, given a critical configuration, we obtain an infinite number of them by applying rotations.
    As our method relies on counting solutions, we need a system with a finite number of solutions.
	For this, we consider another system of equations, where variables are the dot products of pairs of points: $$ x_{ij} := w_i^T w_j, \quad i \neq j .$$
	Note that these variables represent the cosine of the angles between pairs of points.
	To obtain the new system of equations, we first take dot product in equation \eqref{eq:lagrange_gradient_null}, with respect to each point $w_i$:
	$$ \sum_{j=1, j \neq k}^{n} \frac{ w_i^T w_k - w_i^T w_j }{ \| w_k - w_j \|^2 } = \frac{n-1}{2} w_i^T w_k, \quad \forall \ k,i=1,\ldots,n .$$
    As $w_l^T w_l = 1, \forall \ l=1,\ldots,n$, we have: $\| w_k - w_j \|^2 = 2 (1-w_k^T w_j)$. Using this, the equations may be written in the new variables as:
    \begin{equation} \label{eq:lagrange_gradient_null_xij}
		\sum_{j=1, j \neq k}^{n} \frac{ x_{ik} - x_{ij} }{ 1 - x_{kj} } = (n-1) x_{ik}, \quad \forall \ k \neq i .
	\end{equation}
    Note that for $k=i$ the equation is trivial. We then consider Equation \eqref{eq:center_mass_null}, were we also take dot product with respect to each point $w_j$:
	$$ \sum_{k=1}^{n} w_j^T w_k = 0, \quad \forall \ j=1,\ldots,n .$$	
	Using $w_j^T w_j = 1$, the equations in the new variables are:
	\begin{equation} \label{eq:center_mass_null_xij}
		1 + \sum_{k=1, k \neq j}^{n} x_{jk} = 0, \quad \forall \ j=1,\ldots,n .	
	\end{equation}
	Equations \eqref{eq:lagrange_gradient_null_xij} and \eqref{eq:center_mass_null_xij} determine the new system of equations in the variables $x_{ij}$. The condition $w_i \in S^{d-1}$ of equation \eqref{eq:shpere} is now implicit, and we will not define variables $x_{ii}$.

\subsection{Relation between solutions of both systems} \label{Sec. Relation between solutions of both systems}
    
    Any solution in the variables $w_i$ has an associated solution $x_{ij} = w_i^T w_j$. The reciprocal is also true, provided that we consider complex solutions. That is: given a solution $x_{ij} \in \mathbb{C}$, there is a solution $w_i \in \mathbb{C}^d$, with $x_{ij}=w_i^T w_j$. To see this, we first define two matrices.
    
    \begin{definition}[Dot product matrix] \label{def:dot_product_matrix}
        Let $x_{ij} \in \mathbb{C}, \ 1 \leq i < j \leq n$. Define $X \in \mathbb{C}^{n \times n}$, symmetric, with $X_{ij} = x_{ij}, \ \forall \ i < j$, $X_{ii}=1, \ \forall \ i$.
    \end{definition}
    
    \begin{definition}[Cartesian coordinates matrix]
        Given $n$ vectors $w_i \in \mathbb{C}^d$, define $W \in \mathbb{C}^{d \times n}$, with $w_i$ as column $i$.
    \end{definition}
    
    We now use a corollary of the following well-known result.
    
    \begin{proposition}[Autonne-Takagi factorization {\cite[Corollary 2.6.6]{horn2013}}]
        If $X \in \mathbb{C}^{n \times n}$ is symmetric, there is a unitary $P \in \mathbb{C}^{n \times n}$, and a non-negative diagonal matrix $D \in \mathbb{R}^{n \times n}$, such that: $X = P^T D P$.
        Furthermore, the entries of $D$ are the singular values of $X$.
    \end{proposition}
    
    \begin{corollary} \label{cor:complex_solution}
        If $X \in \mathbb{C}^{n \times n}$ is symmetric with rank $d$, and ones on its diagonal, there exists $W \in \mathbb{C}^{d \times n}$, such that: $X=W^T W$, and $w_i^T w_i = 1, \ \forall \ i$.
    \end{corollary}
    
    \begin{proof}
        As $X$ is symmetric: $X = P^T D P$. Take: $W = \sqrt{ \hat{D} } P$; where $\hat{D} \in \mathbb{R}^{d \times n}$ is the submatrix of $D$ with the positive singular values.
    \end{proof}

    We call $X$ the dot product matrix, as we may write $X = W^T W$.
    Thus, given a solution $X \in \mathbb{C}^{n \times n}$ of rank $d$, we obtain a configuration $(w_1, ..., w_n)$ given by the columns of $W \in \mathbb{C}^{d \times n}$. As the relation between both equation systems is obtained by taking dot product with $w_i$, $i=1,...,n$, and this set is a generator of $\mathbb{C}^d$, we get that both systems are equivalent.
    Thus, we do not lose or introduce solutions in $\mathbb{C}$.
    The next result characterizes the correspondence between real solutions of both systems, which is the ones we are really interested in.
    
    \begin{proposition}[{\cite[Corollary 2.5.11]{horn2013}}] \label{prop:x_es_dot_prod}
    	A matrix $X \in \mathbb{R}^{n \times n}$ is symmetric and positive semi-definite if and only if there exists $W \in \mathbb{R}^{d \times n}$, such that $X=W^T W$, for some $d \in \mathbb{N}$.
    	In this case, if $X = Q^T D Q$ is an orthogonal eigenvalue factorization of $X$, we can take $W := \sqrt{D} Q$.
    \end{proposition}

    Note that when the previous result applies, $W$ has columns with unit 2-norm, if and only if $X=W^T W$ has ones on its diagonal.

\subsection{System of polynomial equations} \label{sec:system_dot_product}

	To apply tools from Algebraic Geometry, we need to express the new equations as a system of polynomial equations. The condition that points are pairwise distinct is: $x_{kj} \neq 1, \ \forall \ k \neq j$. This may be written introducing auxiliary variables $z_{kj}$, such that:
    \begin{equation} \label{eq:definition_zij}
		z_{kj} \left( 1 - x_{kj} \right) = 1, \quad \forall \ k \neq j .
	\end{equation}
	Using these variables, the equations in \eqref{eq:lagrange_gradient_null_xij} can be rewritten as polynomial equations:
	\begin{equation} \label{eq:lagrange_gradient_null_xij_pol}
		\sum_{j=1, j \neq k}^{n} \left( x_{ik} - x_{ij} \right) z_{kj} = (n-1) x_{ik}, \quad \forall \ k \neq i .
	\end{equation}
	Note that $z_{ij}=z_{ji}$, for all $i \neq j$.
	From now on, we will work with the overdetermined system of polynomial equations given by \eqref{eq:center_mass_null_xij} (null center of mass), \eqref{eq:definition_zij} (auxiliary variables), and \eqref{eq:lagrange_gradient_null_xij_pol} (null gradient of the Lagrangian); expressed in the variables $x_{ij}$ and $z_{ij}$, for $i < j$.
    This system has $m := 2 \binom{n}{2}$ variables, and $n + \binom{n}{2} + n(n-1)$ equations. Table \ref{tab:variables_equations_number} shows these numbers for some values of $n$.
	
	\begin{table}[H]
    \centering
	\caption{Number of variables and equations of the polynomial system given by Equations \eqref{eq:center_mass_null_xij}, \eqref{eq:definition_zij} and \eqref{eq:lagrange_gradient_null_xij_pol}, with $x_{ij}=x_{ji}$.}
	\label{tab:variables_equations_number}
	\begin{tabular}{cccccccc}
		\toprule
		 & $n$ & 3 & 4 & 5 & 6 & 7 & 8 \\
		\midrule
		variables & $2 \binom{n}{2}$ & 6 & 12 & 20 & 30 & 42 & 56 \\
		equations & $n(3n-1)/2$ & 12 & 22 & 35 & 51 & 70 & 92 \\
		\bottomrule
	\end{tabular}
	\end{table}

    {Note that this formulation gives the critical configurations for all sphere dimensions, and not just for $S^2$. However, it is possible to add polynomial constraints to select critical configurations only up to a given sphere dimension $S^{k-1} \subset \mathbb{R}^k$. For example, by adding all $(k+1) \times (k+1)$ minors of the $n \times n$ dot product matrix $X$, which implies $\text{rank}(X) \leq k$. This adds ${n \choose k+1} \times {n \choose k+1}$ equations to the polynomial system, each of degree $k+1$. Table \ref{tab:additional_constraints_rank} shows the number of equations that should be added for different values of $n$ and $k \leq n-1$. We will not include these constraints, since we are interested in characterizing the critical configurations in all dimensions.}
    
    \begin{table}[H]
        \centering
		\caption{Number of additional constraints to impose $\text{rank}(X) \leq k$ using all $(k+1) \times (k+1)$ minors of $X$.}
		\label{tab:additional_constraints_rank}
		\begin{tabular}{ccccc}
			\toprule
            & \multicolumn{4}{c}{$k \ / \ \text{rank}(X) \leq k$} \\
			$n$ & 2 & 3 & 4 & 5 \\
			\midrule
			5 & 25 & 1 &  & \\
			6 & 225 & 36 & 1 & \\
			7 & 1225 & 441 & 49 & 1 \\
			\bottomrule
		\end{tabular}
	\end{table}

\subsection{Counting complex solutions} \label{sec:counting_solutions}

	To apply our method, we need to verify that the system of polynomial equations \eqref{eq:center_mass_null_xij}, \eqref{eq:definition_zij} and \eqref{eq:lagrange_gradient_null_xij_pol} has a finite number of solutions, and to count this number.
    We first define the ideal $I \subset \mathbb{Q}[z_{ij},x_{ij}]$, generated by the equations.
    The set of complex solutions of these equations is denoted by $V(I) \subset \mathbb{C}^m$, where $m$ is the number of variables. 
    We then use the following result, which allows us to compute an upper bound on the true number of critical points via Gr\"obner bases.
    
	\begin{theorem}[{\cite[Finiteness Th. p. 39, Corollary 2.5 Ch. 4]{cox2005}}] \label{th:finitness}
		Let $I \subset \mathbb{Q}[x_1,\ldots,x_m]$ be an ideal, and $G$ a Gr{\"o}bner basis of $I$.
		The set of complex solutions $V(I) \subset \mathbb{C}^m$ is finite, if and only if, for every variable $x_i$, there exists an integer $\alpha_i \geq 0$, such that $x_i^{\alpha_i} \in LM(G)$; where $LM(G)$ denotes the set of leading monomials of the elements of $G$.
		In this case, the number of complex solutions, counted with multiplicity, is the number of monomials of the ring that are not in the ideal $\left\langle LM(G) \right\rangle$ (not divisible by any element of $LM(G)$).
	\end{theorem}

	To calculate a Gr{\"o}bner basis we use \emph{msolve} \cite{berthomieu2021}, with graded reverse lexicographic monomial order (grevlex).
    We then pass this basis to \emph{Macaulay2} (\emph{M2}) \cite{M2} and use it to verify that the number of solutions is finite and to obtain the number of solutions $(\deg I)$.
	Both \emph{msolve} and \emph{M2} are open source software.
	\emph{Msolve} implements the F4 algorithm \cite{faugere1999}, with parallel algorithms for solving the linear systems generated by F4.
    Table \ref{tab:complex_solutions_number} shows the {``expected number of solutions'' of the polynomial system for different number of points. This is equal to the exact number of (possibly complex) solutions, counted with multiplicity, as defined in Section \ref{sec:intro}.}

    \begin{table}[H]
        \centering
		\caption{Expected solutions of polynomial system for $n \leq 6$. Time and RAM of \emph{msolve} v.0.73 to calculate a Gr\"obner basis (grevlex) using 20 threads.}
		\label{tab:complex_solutions_number}
		\begin{tabular}{cccccccc}
			\toprule
            & \multicolumn{2}{c}{Time (s) \emph{msolve}} & & \multicolumn{2}{c}{Basis} & \multicolumn{2}{c}{\emph{M2}} \\
			$n$ & Total & Lift to $\mathbb{Q}$ & RAM & \# pols & size & $\dim I$ & $\deg I$ \\
			\midrule
			4 & 0.70 & 0.10 & 3.7 MB & 15 & 600 B & 0 & 4 \\
			5 & 0.68 & 0.25 & 3.7 MB & 130 & 77 KB & 0 & 38  \\
			6 & 3963 & 3469 & 37 GB & 2473 & 1.2 GB & 0 & 938 \\
			\bottomrule
		\end{tabular}
	\end{table}

    {Note that, for counting the number of solutions, we only need the leading monomials of a Gr\"obner basis. We opt to calculate the whole basis, and in particular to lift all the basis coefficients to the rationals, to inform the resources needed, as this basis could be used for other purposes in the stage of finding solutions.} 

\subsection{Counting permutations}

	Once we know the number of expected solutions of our polynomial system, we need to find explicit solutions by other methods, until we match the number of expected solutions.
    
	Notice that given $X=W^T W$, and a permutation matrix $P$, then $P^T X P = (WP)^T (WP)$. Since $WP$ is acting by permutations on the points of the original problem on the sphere, and this last problem is invariant by permutations, we have the following result.

	\begin{proposition}
		Let $(X,Z)$ be a solution of the polynomial system given by \eqref{eq:center_mass_null_xij}, \eqref{eq:definition_zij} and \eqref{eq:lagrange_gradient_null_xij_pol}; where $Z \in \mathbb{C}^{n \times n}$ is symmetric, with $Z_{ij} = z_{ij}, \ \forall \ i < j$, $Z_{ii}=0, \ \forall \ i$.
		For every permutation matrix $P \in \mathbb{R}^{n \times n}$, the conjugate by $P$: $\left( P^T X P, P^T Z P \right)$, is also a solution.
	\end{proposition}
        
	\begin{corollary}
		If $X$ is a solution to the polynomial system, the number of different solutions given by conjugations of $X$ by permutations, is the size of the orbit of $X$ under conjugation by permutations.
	\end{corollary}

    Now we give a method to count the different permuted solutions. Instead of calculating $|Orb(X)|$ directly, we calculate the size of the orbit stabilizer subgroup $|Stab(X)|$, as it is easier to implement.
	Both sizes are related by the Orbit Stabilizer Theorem: $$ |Orb(X)| = \frac{|S_n|}{|Stab(X)|} = \frac{n!}{|Stab(X)|} .$$
	To calculate $|Stab(X)|$ for a given $X$, we iterate through all permutation matrices $P$, to count the permutations that satisfy: $X = P^T X P$.

\subsection{Energies and optimal configurations}

	The product energy, as defined in equation \eqref{eq:fekete_prod}, may be written as: $$ E := \prod_{i=1}^{n} \prod_{j=i+1}^{n} \|w_i - w_j\|^2 = 2^{\binom{n}{2}} \prod_{i=1}^{n} \prod_{j=i+1}^{n} \left( 1 - x_{ij} \right) .$$
    An optimal configuration in $S^{d-1}$, is a critical configuration with the maximum product energy, within all critical configurations of $S^{d-1}$.    
    A configuration $W$ is in $S^{d-1} \subset \mathbb{R}^d$ whenever $d \geq rk(W)$. For example, the Equator has $rk(W)=2$, and it can be embedded in any $S^{d-1}$, with $d \geq 2$.  
    Thus, an optimal configuration in $S^{d-1}$ is a critical configuration $W \in \mathbb{R}^{d \times n}$ with the maximum product energy, within all critical configurations with $rk(X) \leq d$, $X=W^T W$.

\section{Results for $n\leq 7$ points}
    
    In what follows, we apply our method to the different formulations presented above and for different numbers of points. This allows us to recover some known results and to prove new theorems in previously unexplored cases.

\subsection{Four and five points} \label{sec:four_and_five_points}

	We now apply our method for the case of four and five points on the sphere. The case of six points will be treated separately as it will require some additional ideas.
	We need to match the number of solutions of Table \ref{tab:complex_solutions_number} with the solutions we may find explicitly.	

    \vspace{1em}
    
    \textbf{Four points.} In this case we have at least two natural candidates for critical configurations: the regular Tetrahedron and the Equator.
	In the Tetrahedron all pairs of points have the same angle, with inner products: $w_i^T w_j = -\frac{1}{3}, \ \forall \ i \neq j$. The dot product matrix is:
	$$ X_{\text{Tet.}} = \begin{pmatrix}
	   1 & \theta & \theta & \theta \\
	   \theta & 1 & \theta & \theta \\
	   \theta & \theta & 1 & \theta \\
	   \theta & \theta & \theta & 1 \\
	\end{pmatrix}, \quad \theta := -\frac{1}{3} .$$
    For the Equator, an associated dot product matrix is:
	$$ X_{\text{Eq.}} = \begin{pmatrix}
		1 & 0 & -1 & 0 \\
		0 & 1 & 0 & -1 \\
		-1 & 0 & 1 & 0 \\
		0 & -1 & 0 & 1 \\
	\end{pmatrix} .$$
    To verify that each $X$ is a solution, we replace its values in the polynomial equations using \emph{M2}.
    Table \ref{tab:orbit_energy_rank_n4} shows the size of the orbit of each $X$ by conjugation. Remember that the elements of each orbit are different solutions of the polynomial system. Their sum matches the number of expected solutions. This proves that for $n=4$ the only kind of critical configurations are the Tetrahedron and the Equator.
    Table \ref{tab:orbit_energy_rank_n4} also shows the Energy and rank associated to each critical configuration. We see that the optimal configurations for $n=4$ are: the Equator in $S^1$ and the Tetrahedron in $S^2$.

	\begin{table}[H]
		\centering
		\caption{$n=4$. Orbit size, energy and rank.}
		\label{tab:orbit_energy_rank_n4}
		\begin{tabular}{cccc}
			Conf. & $|Orb(X)|$ & Energy $E / 2^{n \choose 2}$ & $rank(X)$ \\
			\midrule
			Tetrahedron & 1 & $\simeq 5.619$ & 3 \\
			\midrule
			Equator & 3 & 4.0 & 2 \\
			\midrule\midrule
			Found solutions & 4 &  &  \\
			\cmidrule(lr){1-2}
			Expected solutions & 4 &  &  \\
			\bottomrule
		\end{tabular}
	\end{table}

    \vspace{1em}

    \textbf{Five points.} We have at least three candidates for critical configurations in $S^2$: the Equator and two other denoted 1:3:1 and 1:4. (F\"oppl notation \cite{foppl1912}).
	Configuration 1:3:1 has two points forming a dipole, and three other points equidistributed in a plane through the origin, orthogonal to the dipole. The dot product matrix is:
	$$ X_{1:3:1} = \begin{pmatrix}
		1 & A & 0 & 0 & 0 \\
		A & 1 & 0 & 0 & 0 \\
		0 & 0 & 1 & B & B \\
		0 & 0 & B & 1 & B \\
		0 & 0 & B & B & 1
	\end{pmatrix}, \quad \begin{array}{l} A := -1, \\ B := -\frac{1}{2} \end{array} .$$
	Configuration 1:4 may be represented as the north pole, and four other points equidistributed in a plane orthogonal to the $z$-axis. Using the condition of null center of mass, this plane is $z = -\frac{1}{4}$.
	The associated dot product matrix is:
	$$ X_{1:4} = \begin{pmatrix}
        1 & A & A & A & A \\
        A & 1 & B & C & B \\
        A & B & 1 & B & C \\
        A & C & B & 1 & B \\
        A & B & C & B & 1 \\
	\end{pmatrix}, \quad \begin{array}{l} A := -\frac{1}{4}, \\ B := \frac{1}{16}, \\ C := -\frac{7}{8} \end{array} .$$
    For the Equator, the dot product matrix is:
    $$ X_{\text{Eq.}} = \begin{pmatrix}
	   1 & A & B & B & A \\
	   A & 1 & A & B & B \\
	   B & A & 1 & A & B \\
	   B & B & A & 1 & A \\
	   A & B & B & A & 1
	\end{pmatrix}, \quad \begin{array}{l} A := \frac{-1+\sqrt{5}}{4}, \\ B := \frac{-1-\sqrt{5}}{4} \end{array} .$$
	We also consider a critical configuration on $S^3$, given by the 4-simplex; where any pair of points have the same dot product $\theta$. The condition of null center of mass implies: $\theta = -\frac{1}{4}$.
    
	Table \ref{tab:orbit_energy_rank_n5} shows the size of the orbit of each $X$ by conjugation. Their sum matches the number of expected solutions. This proves that for $n=5$ the only kind of critical configurations are those described in this section.
    Table \ref{tab:orbit_energy_rank_n5} also shows the Energy and rank associated to each critical configuration.
    We see that the optimal configurations for $n=5$ are: the Equator in $S^1$, 1:3:1 in $S^2$, and the 4-simplex in $S^3$. This agrees with already known results. In future sections we will classify the rest of the critical configurations.
    
	\begin{table}[H]
        \centering
		\caption{$n=5$. Orbit size, energy and rank.}
		\label{tab:orbit_energy_rank_n5}
		\begin{tabular}{cccc}
			Configuration & $|Orb(X)|$ & Energy $E / 2^{n \choose 2}$ & $rank(X)$ \\
			\midrule
            4-simplex & 1 & $\simeq 9.313$ & 4 \\
			\midrule
			1:3:1 & 10 & 6.75 & 3 \\
            1:4 & 15 & $\simeq 6.63$ & 3 \\
			\midrule
			Equator & 12 & $\simeq 3.052$ & 2 \\
			\midrule\midrule
			Found solutions & 38 & & \\
			\cmidrule(lr){1-2}
			Expected solutions & 38 & & \\
			\bottomrule
		\end{tabular}
	\end{table}

\subsection{Six points} \label{sec:six_points}

	For six points, the number of expected solutions is 938 (Table \ref{tab:complex_solutions_number}).
	On the other hand, we may imagine the following critical configurations on $S^2$: the Equator, 1:5 and 1:4:1; with dot product matrices:
	$$ X_{1:5} = \begin{pmatrix}
        1 & A & A & A & A & A \\
        A & 1 & B & C & C & B \\
        A & B & 1 & B & C & C \\
        A & C & B & 1 & B & C \\
        A & C & C & B & 1 & B \\
        A & B & C & C & B & 1
	\end{pmatrix}, \quad \begin{array}{l} A := -\frac{1}{5}, \\ B := \frac{ -5 + 6 \sqrt{5} }{25}, \\ C := \frac{ -5 - 6 \sqrt{5} }{25} \end{array} .$$
    $$ X_{1:4:1} = \begin{pmatrix}
        1 & A & 0 & 0 & 0 & 0 \\
        A & 1 & 0 & 0 & 0 & 0 \\
        0 & 0 & 1 & 0 & A & 0 \\
        0 & 0 & 0 & 1 & 0 & A \\
        0 & 0 & A & 0 & 1 & 0 \\
        0 & 0 & 0 & A & 0 & 1
    \end{pmatrix}, \quad A := -1 .$$
    $$ X_{\text{Eq.}} = \begin{pmatrix}
        1 & A & B & C & B & A \\
        A & 1 & A & B & C & B \\
        B & A & 1 & A & B & C \\
        C & B & A & 1 & A & B \\
        B & C & B & A & 1 & A \\
        A & B & C & B & A & 1
	\end{pmatrix}, \quad \begin{array}{l} A := \frac{1}{2}, \\ B := -\frac{1}{2}, \\ C := -1 \end{array} $$
    We also consider the 5-simplex on $S^4$, where all dot products are $\theta = -\frac{1}{5}$.
    Finally, we consider configuration 3:3 on $S^2$, which consists of two parallel planes, each with three equidistributed points, and with the two triangles in phase with each other. The case where triangles have a phase shift of $\frac{\pi}{3}$ is the same as configuration 1:4:1.
    The analysis of 3:3 (with null phase shift) will reveal another critical configuration, which we call ``real conjugate'' of 3:3.

\subsubsection{Configuration 3:3 and its real conjugate}

    Consider configuration 3:3 with null phase shift. This configuration is critical when the triangles are placed at $z = \pm z_0$, where $z_0 := \sqrt{ \frac{ -3 + 2 \sqrt{6} }{5} } \simeq 0.62$.
    To express it in cartesian coordinates, we define the radius of each triangle $R := \sqrt{1-z_0^2}$, and take:
    $$ w_1 = \left( R, 0, -z_0 \right), \ w_2 = \left( -\frac{R}{2}, \frac{R \sqrt{3}}{2}, -z_0 \right), \ w_3 = \left( -\frac{R}{2}, -\frac{R \sqrt{3}}{2}, -z_0 \right) .$$
    $$ w_4 = \left( R, 0, z_0 \right), \ w_5 = \left( -\frac{R}{2}, \frac{R \sqrt{3}}{2}, z_0 \right), \ w_6 = \left( -\frac{R}{2}, -\frac{R \sqrt{3}}{2}, z_0 \right) .$$
    The corresponding dot product matrix is:        
	$$ X_{3:3} = \begin{pmatrix}
		1 & A & A & B & C & C \\
		A & 1 & A & C & B & C \\
		A & A & 1 & C & C & B \\
	    B & C & C & 1 & A & A \\
        C & B & C & A & 1 & A \\
		C & C & B & A & A & 1
	\end{pmatrix}, \quad \begin{array}{l}
    A := \frac{ -7 + 3 \sqrt{6} }{5} \\
    B := \frac{ 11 - 4 \sqrt{6} }{5} \\
    C := \frac{ -1 - \sqrt{6} }{5}
	\end{array} .$$

    To verify that $X_{3:3}$ is critical, we replace its values in the polynomial equations using \emph{M2}.
    As the matrix depends on $\sqrt{6}$, we define a polynomial ring with coefficients in $k:=\mathbb{Q}[t]/ (t^2-6)$. This introduces a variable $t$, such that $t^2=6$.
    We then replace $\sqrt{6}$ by $t$ in $X_{3:3}$, and check that its entries verify the equations in $k[x_{ij},z_{ij}]$.
    As this is true for every $t$ such that $t^2=6$, it implies that if we change $\sqrt{6}$ by $-\sqrt{6}$ in $X_{3:3}$, we also obtain a solution to the polynomial equations. We call this the ``real conjugate'' of 3:3, with matrix:
	$$ X_{\overline{3:3}} = \begin{pmatrix}
		1 & \overline{A} & \overline{A} & \overline{B} & \overline{C} & \overline{C} \\
		\overline{A} & 1 & \overline{A} & \overline{C} & \overline{B} & \overline{C} \\
		\overline{A} & \overline{A} & 1 & \overline{C} & \overline{C} & \overline{B} \\
	    \overline{B} & \overline{C} & \overline{C} & 1 & \overline{A} & \overline{A} \\
        \overline{C} & \overline{B} & \overline{C} & \overline{A} & 1 & \overline{A} \\
		\overline{C} & \overline{C} & \overline{B} & \overline{A} & \overline{A} & 1
    \end{pmatrix}, \quad \begin{array}{l}
    \overline{A} := \frac{ -7 - 3 \sqrt{6} }{5} \simeq -2.87 \\
    \overline{B} := \frac{ 11 + 4 \sqrt{6} }{5} \simeq 4.16 \\
    \overline{C} := \frac{ -1 + \sqrt{6} }{5} \simeq 0.29
	\end{array} .$$
    As $X_{\overline{3:3}}$ has entries with absolute value greater than one, it is not the dot product of points on a unit sphere (it corresponds to a complex critical configuration).
    Note also that we already encountered configurations with non-rational values: the Equator for $n=5$, and 1:5; both depending on $\sqrt{5}$. However, in those cases, changing $\sqrt{5}$ by $-\sqrt{5}$ gives a matrix in the same orbit as the original.
    
	Table \ref{tab:solutions_permutations_n6_inicial} shows the orbit size of each of the ``imaginable'' solutions. Their sum is 268, which is far from the expected 938 solutions.
	As we will show, this is because the polynomial system has other types of solutions, which we are not able to imagine a priori.
		
	\begin{table}[H]
        \centering
		\caption{$n=6$. Number of ``imaginable'' solutions.}
		\label{tab:solutions_permutations_n6_inicial}
		\begin{tabular}{cc}
			\toprule
			Configuration & $|Orb(X)|$ \\
			\midrule
			Equator & 60 \\
			1:5 & 72 \\
			1:4:1 & 15 \\
			3:3 & 60 \\
			3:3 real conj. & 60 \\
			5-simplex & 1 \\
			\midrule\midrule
			Found solutions & 268 \\
			\midrule
			Expected solutions & 938 \\
			\bottomrule
		\end{tabular}
	\end{table}

\subsubsection{Possible values for the dot product} \label{sec:possible_values}

	To identify new solutions, we first find all possible values for one of the variables, say $x_{45}$, and then replace each value in the ideal equations, to find the solutions associated with that value.
	Note that, each time we fix a value for $x_{45}$, we have to solve a polynomial system with much fewer solutions than the original.
    Also, as the problem is invariant under permutations, knowing the values of $x_{45}$ is the same as knowing the possible values of any variable $x_{ij}$.
    
	To find all possible values for $x_{45}$, we calculate a Gr{\"o}bner basis with a monomial order that eliminates all variables, except $x_{45}$. This gives a reduced Gr{\"o}bner basis for the ideal $I \cap \mathbb{Q}[x_{45}]$. As this is a principal ideal domain, its basis is a set with only one polynomial. The roots of this polynomial are the possible values for $x_{45}$.   
	We use \emph{M2} to calculate the factors of the generator of $I \cap \mathbb{Q}[x_{45}]$. Table \ref{tab:generator_factors_n6} shows these factors and their roots.
    Calculating the Gr\"obner basis with \emph{msolve} takes 10 hours and 190 GB of RAM, using 20 threads.

	\begin{table}[H]
    \centering
	\caption{$n=6$. Factors of the generator of $I \cap \mathbb{Q}[x_{45}]$.}
	\label{tab:generator_factors_n6}
	\begin{tabular}{ccc}
		\toprule
		& Generator Factor & Roots \\
		\midrule
		1 & $ x_{45}$ & 0 \\
		2 & $(x_{45} + 1)^2$ & -1 \\
		3 & $ 2 x_{45} - 1 $ & 1/2  \\
		4 & $ 2 x_{45} + 1 $ & -1/2 \\
		5 & $ ( 5 x_{45} - 1 )^2 $ & 1/5 \\
		6 & $ 5 x_{45} + 1 $ & -1/5 \\
		7 & $ ( 5 x_{45} + 7 )^2 $ & -7/5 \\
		8 & $ ( 5 x_{45}^2 + 1 )^2 $ & $\pm \frac{i}{\sqrt{5}}$ \\
		9 & $ 5 x_{45}^2 - 22 x_{45} + 5 $ & $\frac{ 11 \pm 4 \sqrt{6} }{5}$ \\
		10 & $ 5 x_{45}^2 + 2 x_{45} - 1 $ & $\frac{ -1 \pm \sqrt{6} }{5}$ \\
		11 & $ 5 x_{45}^2 + 14 x_{45} - 1 $ & $\frac{ -7 \pm 3 \sqrt{6} }{5}$ \\
		12 & $ 25 x_{45}^2 + 28 x_{45} + 19 $ & $\frac{ -14 \pm i 3 \sqrt{31} }{25}$ \\
		13 & $ 125 x_{45}^2 + 50 x_{45} - 31 $ & $\frac{ -5 \pm 6 \sqrt{5} }{25}$ \\
		14 & $ 100 x_{45}^4 + 95 x_{45}^3 - 21 x_{45}^2 - 22 x_{45} + 10 $ & 4 complex \\
		15 & $ 250 x_{45}^4 + 110 x_{45}^3 - 21 x_{45}^2 - 19 x_{45} + 4 $ & 4 complex \\
		16 & $ 400 x_{45}^4 + 488 x_{45}^3 - 111 x_{45}^2 - 196 x_{45} + 67 $ & 4 complex \\
		17 & $3 x_{45} + 1$ & -1/3 \\
		18 & $5 x_{45} + 4$ & -4/5 \\
		19 & $10 x_{45} - 1$ & 1/10 \\
		20 & $25 x_{45} - 1$ & 1/25 \\
		21 & $25 x_{45} + 11$ & -11/25 \\
		22 & $25 x_{45} + 23$ & -23/25 \\
		\bottomrule
	\end{tabular}
	\end{table}	

    {Note that finding a Gr\"obner basis with an elimination order is usually more expensive than using a grevlex order, both in execution time and memory usage. However, since our ideal is zero dimensional, the projection could be done with any Gr\"obner basis, by means of working in the finite dimensional vector space $R/I$ \cite{cox2005}[Ch. 2, Exercise 2]. As an alternative, a change of monomial order could be applied efficiently by using the FGLM algorithm (for example from grevlex to lex) \cite{faugere_1993}. In particular, the same Gr\"obner basis used to count solutions could be used to find the possible values of the solutions coordinates, without too significant additional effort.}

\subsubsection{Minimal primes and new solutions} \label{sec:minprimes_proc}

	For each factor $h_i$ of the generator of the ideal $I \cap \mathbb{Q}[x_{45}]$, denote by $\sqrt{h_i}$ its square-free part. We consider the ideal $I_i := I + ( \sqrt{h_i} )$, obtained by adding $\sqrt{h_i}$ to the generators of $I$.
    For each $I_i$ we calculate a Gr\"obner basis with \emph{msolve}, using grevlex monomial order, and then use it as input of \emph{M2} to factorize $I_i$ into its minimal primes: $$ I_i = J_1 \cap J_2 \cap \ldots \cap J_k ;$$ where each $J_l$ is a minimal prime ideal of $I_i$.
	The solutions of $I_i$ are: $$ V(I_i) = V(J_1) \cup V(J_2) \cup \ldots \cup V(J_k) .$$
	Thus, we may find the solutions $V(I_i)$ by solving the polynomial system of each ideal $J_l$.
	Appendix \ref{sec:minprimes_sols} shows the results for the factors that give new solutions. There are seven new solutions, some of them with complex coordinates.
    Execution time and RAM is negligible, except for Factor 12, where \emph{msolve} takes 253s and 16 GB.

\subsubsection{Orbits including new solutions}

	Table \ref{tab:solutions_permutations_n6} shows the critical configurations identified for six points and the size of each orbit by conjugation. There are still 90 solutions remaining to match the 938 expected solutions.
	As we show in the next section, this is because ``Complex 1'' has multiplicity greater than one. This, together with the value of expected number of solutions, implies that ``Complex 1'' has multiplicity exactly two, and that we have found all critical configurations for six points.

	\begin{table}[H]
        \centering
		\caption{$n=6$. Orbit size and expected number of solutions.}
		\label{tab:solutions_permutations_n6}
		\begin{tabular}{cc}
			\toprule
			Configuration & $|Orb(X)|$ \\
			\midrule
			Equator & 60 \\
			1:5 & 72 \\
			1:4:1 & 15 \\
			3:3 (and real ``conjugate'') & 60 + 60 \\
			Complex 1 & 90 \\
			Complex 2 & $360 = 180 \times 2$ \\
			{Real 1} & 15 \\
			{Real 2} & 45 \\
			{Real 3} & 60 \\
			{Real 4} & 10 \\
			5-simplex & 1 \\
			\midrule
			Found solutions & 848 \\
			\midrule
			Expected solutions & 938 \\
			\midrule
			Difference & $938 - 848 = 90$ \\
			\bottomrule
		\end{tabular}
	\end{table}

\subsubsection{Multiplicity of ``Complex 1''}
	
	\begin{proposition}[{\cite[Theorem 5.1]{hartshorne2013algebraic}}]
		Consider an ideal $I \subset \mathbb{Q}[x_1,\ldots,x_m]$, generated by polynomials $\{ g_1, \ldots, g_r \}$.
		Define the Jacobian matrix of the ideal generators as: $$ J (x) \in \mathbb{R}^{r \times m} \quad / \quad J(x)_{ij} := \frac{\partial g_i}{\partial x_j}(x), \quad \forall \ x \in \mathbb{C}^m .$$
		Let $\hat{x}$ be a solution to the system of equations associated with the ideal generators.
		This solution has multiplicity greater than one, if and only if, the matrix $J(\hat{x})$ is rank deficient.
	\end{proposition}
  
	For ``Complex 1'', $J(\hat{x},\hat{z})$ has size $30 \times 51$ and rank 29. This implies that ``Complex 1'' has multiplicity greater than one, as we wanted to prove. Calculations are done with \emph{M2} code \emph{solMultiplicty.m2}.

\subsubsection{Cartesian coordinates on the sphere} \label{sec:cartesian_coords_sphere}
    
	Table \ref{tab:eigenvalues_x_n6} shows the eigenvalues of each solution.
    When $X$ is real, symmetric and positive semi-definite, we use Proposition \ref{prop:x_es_dot_prod} to obtain a critical configuration on the unit sphere (Appendix \ref{sec:cartesian_coord_new}). 
    For the other cases, although we cannot obtain a real solution on the sphere, we can obtain a complex solution applying Corollary \ref{cor:complex_solution}.
    This shows that Problem \eqref{eq:fekete_log} has three complex critical configurations.

	\begin{table}[H]
        \centering
		\caption{$n=6$. Properties of solution matrix $X$.}
		\label{tab:eigenvalues_x_n6}
		\begin{tabular}{ccccc}
			Conf. & $X$ real & $X$ eigenvalues $\neq 0$ & $X \succeq 0$ & $rk(X)$ \\
			\midrule
			Equator & yes & 3 $(\times 2)$ & yes & 2 \\
			1:5 & yes & $\frac{6}{5}$ $(\times 1)$, $\frac{12}{5}$ $(\times 2)$ & yes & 3 \\
			1:4:1 & yes & 2 $(\times 3)$ & yes & 3 \\
			3:3 ($\sqrt{6}$) & yes & $\simeq 2.28 (\times 1)$, $\simeq 1.86 (\times 2)$ & yes & 3 \\
			3:3 ($-\sqrt{6}$) & yes & $\simeq -9.48 (\times 1)$, $\simeq 7.74 (\times 2)$ & no & 3 \\
			Complex 1 & no & $\frac{6}{5} (\times 1)$, $\frac{12}{5} (\times 2)$ & yes & 3 \\
			Complex 2 & no & 3 complex different & no & 3 \\
			{Real 1} & yes & $\frac{4}{3} (\times 3)$, $2 (\times 1)$ & yes & 4 \\
			{Real 2} & yes & $\frac{6}{5} (\times 2)$, $\frac{9}{5} (\times 2)$ & yes & 4 \\
			{Real 3} & yes & $\frac{6}{5} (\times 1)$, $\frac{36}{25} (\times 2)$, $\frac{48}{25} (\times 1)$ & yes & 4 \\
			{Real 4} & yes & $\frac{3}{2} (\times 4)$ & yes & 4 \\
            5-simplex & yes & $\frac{6}{5} (\times 5)$ & yes & 5 \\
			\bottomrule
		\end{tabular}
	\end{table}

\subsubsection{Energies and solution in $S^3$ for six points}

	Table \ref{tab:energy_rank_n6} shows the Energy and rank associated to each critical configuration that can be expressed as $X = W^T W$, with $W \in \mathbb{R}^{rk(X) \times n}$. 
    We conclude that the optimal configurations for $n=6$ are: the Equator in $S^1$, 1:4:1 in $S^2$, {Real 4} in $S^3$, and the 5-simplex in $S^4$.
    
	\begin{table}[H]
        \centering
		\caption{$n=6$. Energy, rank and orbit size.}
		\label{tab:energy_rank_n6}
		\begin{tabular}{cccc}
			Configuration & Energy $E / 2^{n \choose 2}$ & $rk(X)$ & |Orb(X)| \\
			\midrule
            5-simplex & $\simeq 15.41$ & 5 & 1 \\
			\midrule
            {Real 4} & $\simeq 11.39$ & 4 & 10 \\
            {Real 1} & $\simeq 11.24 $ & 4 & 15 \\
			{Real 3} & $\simeq 11.17 $ & 4 & 60 \\
			{Real 2} & $\simeq 10.97 $ & 4 & 45 \\
			\midrule
			1:4:1 & $ 8.00 $ & 3 & 15 \\
            3:3 ($\sqrt{6}$) & $ \simeq 6.62 $ & 3 & 60 \\
            3:3 ($-\sqrt{6}$) & $\notin S^2$ & 3 & 60 \\
            1:5 & $ \simeq 5.05 $ & 3 & 72 \\
            Complex 1 & non real coord. & 3 & 90 \\
            Complex 2 & non real coord. & 3 & 360 \\
			\midrule
			Equator & $ \simeq 1.42 $ & 2 & 60 \\
			\bottomrule
		\end{tabular}
	\end{table}	

    For each sphere $S^{d-1}$, the optimal configuration $X$ has the maximum number of conjugate symmetries (minimum orbit size).
    This is also true for $n=4$ and $n=5$, as can be seen in Tables \ref{tab:orbit_energy_rank_n4} and \ref{tab:orbit_energy_rank_n5}.
    Also, within each sphere $S^{d-1}$, the product energy increases as the conjugate symmetries increase (decreases with the orbit size); the only exception being {Real 2} and {Real 3}.    

\subsection{Geometric interpretation on $S^3$} \label{sec: Geometric interpretation on S^3}

	{Real 4}, which is the solution on $S^3 \subset \mathbb{R}^4$ for six points, consists of two equilateral triangles, each inscribed in a copy of $S^1$ lying in orthogonal spaces.
    In particular, from the cartesian coordinates given in Equation \eqref{eq:cartesian_new4} of the appendix, we see that points $w_0$, $w_1$ and $w_2$ form an equilateral triangle, inscribed in an $S^1$. This is because they have the same pairwise angles: $\cos^{-1}(-1/2) = \frac{2 \pi}{3}$, and belong to the same plane through the origin: $w_0 + w_1 = -w_2$. The same happens with points $w_3$, $w_4$ and $w_5$, which form an equilateral triangle, inscribed in another $S^1$. Finally, these circumferences lie in orthogonal spaces: $$ w_i^T w_j = 0, \ \forall \ i \in \{0,1,2\}, \ j \in \{3,4,5\} .$$

    The other critical configurations can be viewed as given by the fibration of $S^3$ by spheres, as follows. 
    {Real 1} is, in fact, analogous to the 1:4:1 configuration of $S^2$: it has the poles and the optimal configuration for 4 points in the equatorial sphere.
    {Real 3} is, in turn, analogous to the 1:5 configuration of $S^2$: it has a pole and the optimal configuration for 5 points in the corresponding sphere.
    {Real 2} has no analogous configuration on $S^2$. It has 4 points on the Equator of a sphere and the optimal configuration for 2 points in another sphere, with the peculiarity that the line that passes through these two points is orthogonal to the plane of the 4 point Equator.

\subsection{Classification of critical configurations}

	We now classify all critical configurations, using the Hessian of the Lagrangian.
	First consider the case of the unit sphere $S^2 \subset \mathbb{R}^3$.
	The gradient of the Lagrangian of Problem \eqref{eq:fekete_log} has coordinate functions:
	$$ \frac{\partial L}{\partial w_k} = - \sum_{j=1, j \neq k}^{n} \frac{ 2 \left( w_k - w_j \right) }{ \| w_k - w_j \|^2 } + (n-1) w_k, \quad \forall \ k=1,\ldots,n .$$
	These may be written as:
	$$ \frac{\partial L}{\partial w_k} = \sum_{j=1, j \neq k}^{n} f(w_k-w_j) + (n-1) w_k, \quad \forall \ k=1,\ldots,n ,$$
	where we define the auxiliary function $f: \mathbb{R}^3 \to \mathbb{R}^3$, such that $$ f(w) := - \frac{2 w}{ \| w \|^2 } = -\frac{ 2 (x,y,z) }{ x^2 + y^2 + z^2 }, \quad w := (x,y,z) .$$
    The Jacobian of $f$ is:
	$$ J_f(w) := \frac{-2}{\| w \|_2^4} \begin{pmatrix}
		-x^2 + y^2 + z^2 & -2xy & -2xz \\
		 -2xy & x^2 - y^2 + z^2 & -2yz \\
		-2xz & -2yz & x^2 + y^2 - z^2 \\
	\end{pmatrix} .$$
	Using this matrix, we can write the Hessian $H_L$ of the Lagrangian as $n^2$ blocks of $3 \times 3$ matrices, given by: 
	$$ (H_L)_{ii} := \frac{\partial L}{\partial w_i^2} = \sum_{j=1, j \neq i} J_f(w_i-w_j) + (n-1) \textrm{Id}_3, \quad \forall \ i .$$	
	$$ (H_L)_{ij} := \frac{\partial L}{\partial w_i w_j} = - J_f(w_i-w_j), \quad \forall \ i \neq j .$$
    
	These expressions also apply to points on $S^{d-1} \subset \mathbb{R}^d$, with $H_L \in \mathbb{R}^{d n \times d n}$.
    As we are interested in the eigenvalues of $H_L$ associated with directions of the tangent space of the product of spheres, we project $H_L$ onto this tangent space. The projected matrix is: $$ h_L := V^T H_L V, \quad h_L \in \mathbb{R}^{n(d-1) \times n(d-1)} ;$$
    where the columns of $V$ are an orthonormal basis of the tangent space.
    This basis is obtained as the orthogonal complement of a basis of the normal space: $B_n = \{v_1,\ldots,v_n\}$; where $v_i = \left( \vec{0}, w_i, \vec{0} \right) \in \mathbb{R}^{nd}$.
    We use the following result to classify the critical configurations.

    \begin{proposition}
        Consider a critical configuration $w$ of Problem \eqref{eq:fekete_log}. Denote by $h_L$ the Hessian of the Lagrangian at $w$, projected onto the tangent space of the product of $n$ spheres $S^{d-1}$.
        \begin{enumerate}
            \item if $h_L$ has at least one negative eigenvalue, then $w$ is a saddle configuration.

            \item if $h_L$ has eigenvalue zero with multiplicity exactly $d(d-1)/2$, and all other eigenvalues of $h_L$ are positive, then $w$ is a local minimum (possibly global).
            
            \item If $w$ is a saddle configuration in $S^k$, it is also a saddle in $S^l$, $\forall \ l \geq k$.
        \end{enumerate}
    \end{proposition}

    \begin{proof}
    \begin{enumerate}
        \item Problem \eqref{eq:fekete_log} does not admit local maxima. The existence of a negative eigenvalue implies the critical configuration is not a local minima, and so it must be a saddle.
        
        \item As the energy remains constant under rotations, $h_L$ will always have zero as an eigenvalue, with multiplicity at least the dimension of the orthogonal group: $|O(d)| = d(d-1)/2$. When this inequality is exact, zero is associated to rotations only. If the other eigenvalues are positive, then the Hessian classifies and the critical configuration is a local minima.

        \item If $w$ is a saddle point in $S^k$, there exists a local direction in which the energy value increases. This direction remains if we increase the dimension of the sphere.
    \end{enumerate}

    \end{proof}

	The eigenvalues of each $h_L$ are given in Appendix \ref{sec:eigenvalues_hessian}.
    Table \ref{tab:classification_n6} shows the classification of each critical configuration.
    All configurations are saddle points or global minima, except for ``{Real 1}'', which is a local minima in $S^3$, that is not a global minima.
    Also, every saddle point has a negative direction associated with the projected Hessian.
    Note that the Equator is a global minima in $S^1$, and a saddle point in $S^2$. This also happens with other configurations.
    
	\begin{table}[H]
    \centering
	\caption{$n=6$. Classification. S is Saddle, GM is global minima, and SM is local minima that is not global (spurious).}
	\label{tab:classification_n6}
	\begin{tabular}{ccccc}
		Configuration & $S^1$ & $S^2$ & $S^3$ & $S^4$ \\
		\midrule
		Equator & GM & S & S & S \\
		\midrule
		1:5 & - & S & S & S \\
		1:4:1 & - & GM & S & S \\
		3:3 ($\sqrt{6}$) & - & S & S & S \\
		\midrule
		{Real 1} & - & - & SM & S \\
		{Real 2} & - & - & S & S \\
		{Real 3} & - & - & S & S \\
		{Real 4} & - & - & GM & S \\
		\midrule
		5-simplex & - & - & - & GM \\
		\bottomrule
	\end{tabular}
	\end{table}

\subsection{{Seven Points}}

    As we mentioned in Section \ref{sec:related_work}, for seven points there is no known solution for the Fekete problem in $S^2$. To the best of our knowledge, no solution is known for seven points in $S^3$. The solution is already known for $S^4$ \cite{Dragnev2016}[Theorem 1.9] and $S^5$ \cite{kolushov1997}.
    
    We were unable to apply our method directly to the case of seven points, because the computation of the required Gr\"obner bases did not terminate using the time and resources at our disposal. However, we were able to compute the desired Gr\"obner bases if the polynomial system is simplified by adding one additional constraint, namely the presence of at least one dipole. Our main result in this section is that the only critical configuration in $S^2$ having at least one dipole is the 1:5:1.
    Our numerical experiments support the conjecture that this configuration is also the optimal configuration without the additional constraint.
    Our theorem reduces the problem of proving the conjecture, namely that 1:5:1 is the solution of the Fekete problem for seven points in $S^2$, to showing that some energy minimizer must necessarily have at least one dipole.

    \subsubsection{Numerical Experiments}

    We consider the Fekete problem as formulated in Equation \eqref{eq:fekete_log}, with objective function given by the logarithmic energy of $n$ points: $$ E_{\log} := -\log(E) = - \sum_{i=1}^{n} \sum_{j=i+1}^{n} \log \left( \|w_i - w_j\|^2 \right) .$$
    To estimate the local minima of the logarithmic energy in $(S^{d-1})^n$, we apply Riemannian Gradient Descent (RGD) on the manifold given by the product of the $n$ spheres. We use the package \emph{Pymanopt} \cite{pymanopt_2016} to define the manifold and to calculate the projection to the tangent space and the retraction to the manifold. The step size in the tangent space is calculated with our own implementation of the Armijo rule, with parameters $\tau = 0.5$, $r = 1 \times 10^{-4}$ and $\alpha_0 = 0.1$ \cite{boumal_2023}[Alg. 4.2, p. 62].
    The iterations are stopped when the norm of the projected gradient is less than $1 \times 10^{-5}$.
    Algorithm \ref{alg:gradient_descent_pseudocode} gives a pseudocode of the implemented algorithm, which can be found in the notebook \emph{gradient\_descent\_pymanopt}.
    For each sphere $S^{d-1}$ we perform multiple tests, each with a different random initial configuration. We increase the number of tests together with the dimension of the sphere, to take into account that the search space is getting bigger.

    \begin{algorithm}
    \caption{Riemannian Gradient Descent on the product of $n$ spheres $S^{d-1}$ with Armijo rule.}
        \textbf{Input:} Initial configuration $w^0 \in \left( S^{d-1} \right)^n$, $\text{gradTol}$, $r > 0$, $\tau > 0$, $\alpha_0 > 0$, maxIters. Retraction and Projection operators $R_{w}$ and $P_{w}$, respectively.
        
        \textbf{Output:} configuration $w^k$ with $\| P_{w^k} \left( \nabla E_{\log}(w^k) \right) \| \leq \text{gradTol}$.\\
        
        \begin{algorithmic}[1]
        \State{$k \gets 0$, $\alpha \gets \alpha_0$}
        \State{$\nabla_{T} E_{\log} \gets P_{w^k} \left( \nabla E_{\log}(w^k) \right)$} \Comment{Project gradient to tangent space at $w^k$.}
        \While{$k < \text{maxIters and } \| \nabla_{T} E_{\log} \| > \text{gradTol}$}\\

            \While{$E_{\log} \left( R_{w^k} \left( - \alpha \nabla_{T} E_{\log} \right) \right) > E_{\log}(w^k) - r \alpha \| \nabla_{T} E_{\log} \|^2$}
                \State{$\alpha \gets \tau \alpha$}            \Comment{Update stepsize with Armijo rule.}
            \EndWhile \\

        \State{$w^{k+1} \gets R_{w^k} \left( - \alpha \nabla_T E_{\log} \right)$} \Comment{Update configuration using stepsize $\alpha$.} \\

        \State{$\alpha \gets \frac{\alpha}{\tau}$}  \Comment{Initial Armijo stepsize for next iteration.} \\

        \State{$\nabla_{T} E_{\log} \gets P_{w^{k+1}} \left( \nabla E_{\log}(w^{k+1}) \right)$} \Comment{Projected gradient for next iteration.} \\
        
        \State{$k \gets k + 1$}
        
        \EndWhile
        \end{algorithmic}    
        \label{alg:gradient_descent_pseudocode}
    \end{algorithm}
    
    For $S^2$ we perform 1000 tests, and in all cases we obtain a configuration with energy value: $E_{\log} = -16.3649557$, when rounding up to seven digits. For each test, we also check that the estimated solution has a dipole. This, together with Theorem \ref{theorem:dipole_n7} and the termination criteria of the algorithm, implies that each estimated configuration is of the kind 1:5:1. These results suggest the following conjecture, which already appears in the bibliography, although without the details of the numerical experiments in which it is based \cite{beltran2020}[Conjecture 11.1], \cite{constantineau2023}[Table 1.1].

    \begin{conjecture}
        Configuration 1:5:1 is the unique kind of solution of the Fekete problem for $n=7$ points in $S^2$.
    \end{conjecture}

    For $S^3$ we perform 5000 tests, and in all cases we obtain a configuration with the following energy, when rounding up to seven digits: $E_{\log} = -17.1587805$. For each test, we check that the estimated configuration has the following dot product matrix (or a permutation of it):
    $$ W W^T =
    \begin{pmatrix}
        1 & -1 & 0 & 0 & 0 & 0 & 0 \\
        -1 & 1 & 0 & 0 & 0 & 0 & 0 \\
        0 & 0 & 1 & -1 & 0 & 0 & 0 \\
        0 & 0 & -1 & 1 & 0 & 0 & 0 \\
        0 & 0 & 0 & 0 & 1 & -\frac{1}{2} & -\frac{1}{2} \\
        0 & 0 & 0 & 0 & -\frac{1}{2} & 1 & -\frac{1}{2} \\
        0 & 0 & 0 & 0 & -\frac{1}{2} & -\frac{1}{2} & 1 \\
    \end{pmatrix} .$$

    This configuration has two dipoles, formed by the pairs $(w_1, w_2)$, $(w_3, w_4)$, and a regular simplex given by three points $w_5$, $w_6$ and $w_7$. Also, the dipoles are orthogonal to each other and to the regular simplex. Based on the F\"opl notation, we denote this configuration as $4_{S^1} \times 3_{S^1}$, where we use $\times$ since the configuration lies in $S^1 \times S^1$. These experimental results suggest the following conjecture.
    
    \begin{conjecture}
        Configuration $4_{S^1} \times 3_{S^1}$ is the unique kind of solution of the Fekete problem for $n=7$ points in $S^3$.
    \end{conjecture}

    For $S^4$ it is known that the Fekete problem has exactly two local minima, with one of them being a spurious local minima \cite{Dragnev2023}[Theorem 1.1]. In particular, the global minima is given by ``two orthogonal simplexes, an equilateral triangle and a regular tetrahedron, inscribed in a great circle and a great 3-D hypersphere'' \cite{Dragnev2016}[Theorem 1.9]. The spurious local minima is given by a dipole and a regular simplex of five points, orthogonal to each other.
    Our numerical experiments recover both local minima. For this we perform 10.000 tests, and we obtain configurations with two different energy values, as shown in Table \ref{tab:gd_iterations_energy} in the row of $S^4$.
    For each energy value, $E_{\log}^{(1)}$ and $E_{\log}^{(2)}$, we obtain configurations with respective Cartesian coordinate matrices $W_1$ and $W_2$, which correspond to the known local minima, as can be seen in the associated dot product matrices:
    $$ W_1 W_1^T = \begin{pmatrix}
        1 & -\frac{1}{3} & -\frac{1}{3} & -\frac{1}{3} & 0 & 0 & 1 \\
        -\frac{1}{3} & 1 & -\frac{1}{3} & -\frac{1}{3} & 0 & 0 & 0 \\
        -\frac{1}{3} & -\frac{1}{3} & 1 & -\frac{1}{3} & 0 & 0 & 0 \\
        -\frac{1}{3} & -\frac{1}{3} & -\frac{1}{3} & 1 & 0 & 0 & 0 \\
        0 & 0 & 0 & 0 & 1 & -\frac{1}{2} & -\frac{1}{2} \\
        0 & 0 & 0 & 0 & -\frac{1}{2} & 1 & -\frac{1}{2} \\
        0 & 0 & 0 & 0 & -\frac{1}{2} & -\frac{1}{2} & 1 \\
    \end{pmatrix}, \
    W_2 W_2^T = \begin{pmatrix}
        1 & -1 & 0 & 0 & 0 & 0 & 0 \\
        -1 & 1 & 0 & 0 & 0 & 0 & 0 \\
        0 & 0 & 1 & -\frac{1}{4} & -\frac{1}{4} & -\frac{1}{4} & -\frac{1}{4} \\
        0 & 0 & -\frac{1}{4} & 1 & -\frac{1}{4} & -\frac{1}{4} & -\frac{1}{4} \\
        0 & 0 & -\frac{1}{4} & -\frac{1}{4} & 1 & -\frac{1}{4} & -\frac{1}{4} \\
        0 & 0 & -\frac{1}{4} & -\frac{1}{4} & -\frac{1}{4}  & 1 & -\frac{1}{4} \\
        0 & 0 & -\frac{1}{4} & -\frac{1}{4} & -\frac{1}{4} & -\frac{1}{4} & 1 \\
    \end{pmatrix} .$$

    \begin{table}[H]
    \centering
	\caption{RGD for $n=7$ points in $S^{d-1} \subset \mathbb{R}^d$. Iterations and final energy, with the number of times RGD converges to that value (in percentage).}
	\label{tab:gd_iterations_energy}
	\begin{tabular}{cc|ccc|cc}
        \midrule
        & & \multicolumn{3}{c|}{\# RGD Iterations} & \multicolumn{2}{c}{Energy values} \\
        \midrule
		$S^{d-1}$ & \# RGD tests & min. & mean & max. & $E_{\log}^{(1)}$ & $E_{\log}^{(2)}$ \\
		\midrule
		$S^2$ & 1.000 & 20 & 2921 & 3073 & -16.3649557 (100\%) & N/A \\
		$S^3$ & 5.000 & 50 & 72.6 & 145 & -17.1587805 (100\%) & N/A \\
		$S^4$ & 10.000 & 60 & 98.0 & 253 & -17.4985786 (84.2\%) & -17.4806735 (15.8\%) \\
		$S^5$ & 50.000 & 18 & 24.1 & 58 & -17.7932551 (100\%) & N/A \\
		\bottomrule
	\end{tabular}
	\end{table}

    Finally, for seven points in $S^k$, with $k \geq 5$, it is known that the problem has a unique global minima, given by the 6-simplex \cite{kolushov1997}. We did 50.000 tests for $k=5$, and the algorithm always converged to a configuration with the energy of the simplex: $\simeq -17.7932551$. This suggests that there are no spurious local minima in this case.

    \subsubsection{Dipole implies conjectured minima in $S^2$}

    We now prove our main result of this section.
    
    \begin{theorem} \label{theorem:dipole_n7}
        For the Fekete problem with $n=7$ points in $S^2$, the only critical configuration having a dipole is the 1:5:1.
    \end{theorem}
    
    To prove our result, we work with the system of critical configurations in Cartesian coordinates, as obtained in Section \ref{sec:critical_confs}, and before taking dot products. This is useful as we want to work in $S^2$, while the dot product formulation gives critical configurations in any sphere dimension. Specifically, we use the following system of equations, with variables $w_i \in \mathbb{R}^3$ and $z_{ij} \in \mathbb{R}$:
    \begin{equation} \label{eq:system_cartesian}
        w_i^T w_i = 1, \ (n-1) w_i - \sum_{j=1, j \neq i}^n \left( w_i - w_j \right) z_{ij} = \vec{0}, \ \forall \ i; \quad z_{ij} \left( 1 - w_i^T w_j \right) = 1, \ \forall \ i < j .
    \end{equation}
    Table \ref{tab:variables_equations_number_cartesian} shows the number of variables and equations of this system for some values of $n$. Note that this system has orthogonal symmetries that we should remove to apply our method.
    
	\begin{table}[H]
    \centering
	\caption{Number of variables and equations of the system given by Equations \eqref{eq:system_cartesian} in $\mathbb{R}^3$.}
	\label{tab:variables_equations_number_cartesian}
	\begin{tabular}{cccccccc}
		\toprule
		 & $n$ & 3 & 4 & 5 & 6 & 7 & 8 \\
		\midrule
		variables & $3n + \frac{n(n-1)}{2}$ & 12 & 18 & 25 & 33 & 42 & 52 \\
		equations & $4n + \frac{n(n-1)}{2}$ & 15 & 22 & 30 & 39 & 49 & 60 \\
		\bottomrule
	\end{tabular}
	\end{table}

    To search for critical configurations having a dipole, we fix $w_7 = (-1,0,0)$ and $w_6 = (1,0,0)$ in the previous system, forming a dipole between $w_6$ and $w_7$. To remove orthogonal symmetries, we fix the last coordinate of another point at zero: $w_5 = (x,y,0)$. This way of removing orthogonal symmetries is an alternative to what we did in Section \ref{sec:removing_orthogonal_symmetry} when taking dot products, with the difference that now we remain in $S^2$.

	\begin{figure}[H]
		\centering
    	\begin{tikzpicture}[
    		scale=0.5,
    		rotate=-90, 
    		x={(1cm,0cm)},
    		y={(0cm,1cm)},
    		z={(-0.4cm,0.4cm)} 
    		]
    		
    		\draw[thick] (4.5,0,0) -- (4,0,0);
    		\draw[thick, dashed] (4,0,0) -- (-5,0,0);
    		\draw[->, thick] (-4,0,0) -- (-5,0,0) node[below right] {$z$};
    		
    		\draw[thick] (0,-5,0) -- (0,-4,0);
    		\draw[thick, dashed] (0,-4,0) -- (0,4,0);
    		\draw[->, thick] (0,4,0) -- (0,5,0) node[below left] {$x$};
    		
    		\draw[->, thick, dashed] (0,0,-3.5) -- (0,0,5) node[below right] {$y$};	
    		\draw[thick] (0,0,-5) -- (0,0,-3.5);

    		\def\R{4}      
    		\def\Eq{1.3}     
    		\def\rp{4}     
    		
    		\fill[red!20, opacity=0.6] (0,0) circle (\R);
    		
    		\draw[thick] 
    		(\R,0) arc[start angle=0, end angle=-180, x radius=\R, y radius=\Eq];
    		
    		\draw[thick, dashed] 
    		(\R,0) arc[start angle=0, end angle=180, x radius=\R, y radius=\Eq];

    		\draw[thick] 
    		(0,\R) arc[start angle=90, end angle=-90, x radius=\Eq, y radius=\R];
    		
    		\draw[thick, dashed] 
    		(0,\R) arc[start angle=90, end angle=270, x radius=\Eq, y radius=\R];

    		\def\aA{72} 
    		\coordinate (W1) at ({\rp*cos(15)}, -0.3);	
    		\coordinate (W2) at ({\rp*cos(\aA+6)}, 1.3);
    		\coordinate (W3) at ({\rp*cos(\aA*2-5)}, 0.9);
    		\coordinate (W4) at ({\rp*cos(\aA*3+10)}, -0.9);
    		\coordinate (W5) at ({\rp*cos(\aA*4)}, -1.25);	
    		
    		\coordinate (W6) at (0, 4);	
    		\coordinate (W7) at (0, -4);

    		\foreach \p in {W1,W2,W3,W4}{
    			\fill (\p) circle (7pt);
    		}
    		
    		\fill[red] (W5) circle (7pt);
    		
    		\fill[blue] (W6) circle (7pt);
    		\fill[blue] (W7) circle (7pt);

    		\draw[blue, dashed, thick]
    		(W1) -- (W2) -- (W3) -- (W4) -- (W5) -- cycle;
    		
    		\node[below left] at (W1) {$w_1$};
    		\node[above right] at (W2) {$w_2$};
    		\node[above right] at (W3) {$w_3$};
    		\node[above left] at (W4) {$w_4$};
    		\node[above left] at (W5) {$w_5$};
    		
    		\node[above right] at (W6) {$w_6$};
    		\node[above left] at (W7) {$w_7$};
    		
    	\end{tikzpicture}
        \caption{Seven points. One way of obtaining configuration 1:5:1 when $w_6$ and $w_7$ are fixed forming a dipole, and $w_5$ is restricted to $z=0$.}
        \label{fig:n7_cartesian_dipole}
	\end{figure}

    Configuration 1:5:1 is a particular solution of the resulting polynomial system. Moreover, when counting permutations, configuration 1:5:1 appears as a different solution exactly 48 times. This is because the point $w_5$ may be chosen in two different ways: $w_5 = (0,1,0)$ or $w_5 = (0,-1,0)$ (see Figure \ref{fig:n7_cartesian_dipole}). In both cases, the points $w_1$ to $w_4$ may be freely permuted in $4! = 24$ ways. This shows that the system has at least 48 different solutions (without counting multiplicities), all of them permutations of configuration 1:5:1.
    
    On the other hand, we may count the exact number of solutions of the system. For this we first find a Gr\"obner basis with \emph{msolve}, and then pass it to \emph{Macaulay2} to calculate the radical of the ideal. Using Theorem \ref{th:finitness}, we conclude that the number of solutions associated to this last ideal is 48. This is an upper bound on the number of different solutions. As this upper bound matches the lower bound, we conclude that the only kind of critical configuration with a dipole is configuration 1:5:1, as we wanted to prove. We note that the number of solutions of the ideal, before taking the radical, is $192 = 48 \times 4$; which means that each of the 48 solutions has multiplicity four.
    Table \ref{tab:complex_solutions_number_dipole_n7} shows the execution time and RAM needed to calculate a Gr\"obner basis with grevlex monomial order and rational coefficients. The resources needed to calculate the radical of the ideal with \emph{M2}, with a precomputed Gr\"obner basis as input, are negligible.

    \begin{table}[H]
        \centering
		\caption{Ideal with a dipole for $n=7$ in $S^2$, using Cartesian coordinates. Time and RAM of \emph{msolve} v0.94 to calculate a Gr\"obner basis of the ideal with grevlex and 40 threads.}
		\label{tab:complex_solutions_number_dipole_n7}
		\begin{tabular}{cc|ccc|ccc}
			\toprule
            & & \multicolumn{2}{c}{Time (s)} & & \multicolumn{3}{c}{Basis} \\
			\# vars. & \# eqs. & Total & Lift to $\mathbb{Q}$ & RAM & \# pols & \# mons & file size \\
			\midrule
			35 & 47 & 4542 & 806 & 27.2 GB & 198 & 2096 & 27 KB \\
			\bottomrule
		\end{tabular}
	\end{table}

\subsection{{Comparison of Formulations in $S^2$}}
    
    In this section we compare the performance of the formulation in Cartesian coordinates, given by Equations \ref{eq:system_cartesian}, with the one in terms of dot products, given by Equations \eqref{eq:center_mass_null_xij}, \eqref{eq:definition_zij} and \eqref{eq:lagrange_gradient_null_xij_pol}.
    For the cartesian formulation, we remove orthogonal symmetry by fixing $w_n=(-1,0,0)$, and $w_{n-1}=(x,y,0)$. We also need to avoid a dipole between these two points, since this would produce an $S^1$ symmetry with respect to the dipole axis. For this we introduce an auxiliary variable $t$, and the condition: $$ t \left( w_{n-2}^{(1)} - 1 \right) = 1 ;$$
    where $w_{n-2}^{(1)}$ denotes the first coordinate of $w_{n-2}$. This implies $w_{n-2}^{(1)} \neq 1$.
    On the other hand, for the dot product formulation, we add equations to limit the search of critical configurations to $S^2 \subset \mathbb{R}^3$. Specifically, we add all $4 \times 4$ minors of the dot product matrix $X$, which is equivalent to the condition: $\text{rank}(X) \leq 3$.
    Tables \ref{tab:comparison_formulation_cartesian} and \ref{tab:comparison_formulation_dot_product} show the performance of each formulation when calculating the leading monomials of a Gr\"obner basis with grevlex monomial order, using \emph{msolve} v.0.94 in 40 threads.

    \begin{table}[H]
        \centering
		\caption{Cartesian formulation in $S^2$. Computation of leading monomials of a Gr\"obner basis.}
		\label{tab:comparison_formulation_cartesian}
		\begin{tabular}{c|ccc|cc|cc|cc}
			\toprule
            & & & & \multicolumn{2}{c}{\emph{msolve}} & \multicolumn{2}{c}{Basis} & \multicolumn{2}{c}{Ideal} \\
			$n$ & \# vars & \# eqs & max deg & time (s) & RAM & \# pols & \# mons & $\dim$ & $\deg$ \\
			\midrule
		      4 & 15 & 22 & 3 & 0.12 & 3.85 MB & 23 & 76 & 0 & 8 \\
			5 & 22 & 30 & 3 & 1.17 & 3.84 MB & 247 & 6609 & 0 & 120 \\
			6 & 30 & 39 & 3 & 3110 & 14.4 GB & 5681 & 3.268.865 & 0 & 2688 \\
			7 & 39 & 49 & 3 & \multicolumn{6}{c}{NA}  \\
			\bottomrule
		\end{tabular}
	\end{table}

    \begin{table}[H]
        \centering
		\caption{Dot product formulation in $S^2$. Computation of leading monomials of a Gr\"obner basis.}
		\label{tab:comparison_formulation_dot_product}
		\begin{tabular}{c|ccc|cc|cc|cc}
			\toprule
            & & & & \multicolumn{2}{c}{\emph{msolve}} & \multicolumn{2}{c}{Basis} & \multicolumn{2}{c}{Ideal} \\
			$n$ & \# vars & \# eqs & max deg & time (s) & RAM & \# pols & \# mons & $\dim$ & $\deg$ \\
			\midrule
		      4 & 12 & 23 & 4 & 0.10 & 3.67 MB & 15 & 57 & 0 & 4 \\
			5 & 20 & 60 & 4 & 0.66 & 3.66 MB & 130 & 4178 & 0 & 37 \\
			6 & 30 & 276 & 4 & 564 & 3.23 GB & 2037 & 1.288.061 & 0 & 717 \\
			7 & 42 & 1295 & 4 & \multicolumn{6}{c}{NA} \\
			\bottomrule
		\end{tabular}
	\end{table}

    As can be seen from the previous results, for up to six points, the dot product formulation uses less space and time resources than the Cartesian formulation, and the number of solutions is almost 3.8 times less than in the Cartesian formulation. Although time and space resources could change with a change in the monomial order, the difference in the number of solutions will remain; a matter that is relevant when analysing the different solutions in subsequent stages.
    
    Also, we note that counting the different permutations of a given ``kind'' of solution is more difficult in the Cartesian formulation. In particular because the same ``kind'' of solution could appear in different permutation orbits.

\section{Conclusions and Future Work}
    
    In this work we characterize the critical configurations of the logarithmic Fekete problem for up to six points in all dimensions.    
    We proceed by expressing the set of critical configurations up to the action of the orthogonal group as a zero dimensional variety. Simultaneously, we construct several candidates for critical configurations. The proof of the classification theorem is achieved by showing that the number of candidates matches the degree of the variety, when the candidates are counted with multiplicity. Our computation of the degree relies on Gr\"obner bases.

    {For seven points, although we are not able to apply our method to obtain all critical configurations, we can apply it to prove that if the Fekete problem in $S^2$ has a solution with a dipole, then it is given by configuration 1:5:1.}
    {Furthermore, our numerical experiments suggest that 1:5:1 is indeed the minimum energy configuration.}

    Our approach is, in principle, applicable to any number of points. However, in practice, there seems to be a computational bottleneck in the construction of the Gr\"obner bases of the ideal of the critical configurations variety. The extension of this method to characterize optimal configurations for the case of seven points is the subject of ongoing work.

\section*{Acknowledgements}

    We thank Pedro Raigorodsky and Carlos Beltr\'an for useful discussions during the completion of this work.
	We thank ClusterUY for the infrastructure for the numerical experiments.
	Mat\'ias Vald\'es acknowledges support from a PhD grant from Agencia Nacional de Investigaci\'on e Innovaci\'on (ANII) and Comisi\'on Acad\'emica de Posgrado (CAP). Leandro Bentancur acknowledges support from a PhD grant from CAP. Marcelo Fiori and Mauricio Velasco were partially supported by ANII grant FCE-1-2023-1-176172.
	We are grateful to the reviewers and handling Program Committee members of ISSAC 2025
	Conference for their constructive feedback, which contributed to an improved revised manuscript. We
	also thank the reviewers and the Editor of the Journal of Symbolic Computation for their positive
	evaluation and helpful remarks.

%
%
%
%
%
%
%
%

\appendix

\section{Solutions of minimal primes for $n=6$} \label{sec:minprimes_sols}

In this appendix we describe the new solutions obtained with the procedure of Section \ref{sec:minprimes_proc}, for the case of six points.

\subsection{Factor $5 x_{45} - 1$ - Complex 1}
	
	We obtain 12 minimal primes, and a new kind of solution, which we call ``Complex 1'', as some of its coordinates are complex.
	$$ X_{C1} = \begin{pmatrix}
	       1 & -1 & -x_{15} & x_{15} & x_{15} & -x_{15} \\
            & 1 & x_{15} & -x_{15} & -x_{15} & x_{15} \\
            & & 1 & x_{45} & x_{45} & x_{25} \\
            &  & & 1 & x_{25} & x_{45} \\
            &  & & & 1 & x_{45} \\
            &  & & & & 1 \\
	\end{pmatrix}, \quad \begin{array}{l} x_{45} = \frac{1}{5}, \\ x_{25} = -\frac{7}{5}, \\ x_{15} = \pm \frac{i}{\sqrt{5}} \end{array} .$$

\subsection{Factor $25 x_{45}^2 + 28 x_{45} + 19$ - Complex 2}
	
    We obtain 3 minimal primes and a new kind of solution, which we call ``Complex 2'', as some of its coordinates are complex.
    $$ X_{C2} = \begin{pmatrix}
            1 & A & x_{13} & B & C & C \\
    		& 1 & B & x_{13} & C & C \\
    		& & 1 & D & x_{35} & x_{35} \\
    		  & & & 1 & x_{35} & x_{35} \\
    		& & & & 1 & x_{45} \\
    		& & & & & 1 \\
    	\end{pmatrix} ; $$
    $$ \left\lbrace \begin{array}{l}
            25 x_{45}^2+28x_{45}+19 = 0, \quad 8 x_{13}^2-5x_{13}x_{45}+x_{13}-x_{45}-3 = 0 , \\
            20 x_{35}^2+10x_{35}x_{45}+10x_{35}-x_{45}-3 = 0, \\
            A := 2 x_{35} + \frac{3}{8} x_{45} + \frac{1}{8}, \quad B := -x_{13} + \frac{5}{8} x_{45} -\frac{1}{8}, \\
            C := -x_{35} - \frac{1}{2} x_{45} - \frac{1}{2}, \quad D := -2 x_{35} - \frac{5}{8} x_{45} - \frac{7}{8} .
        \end{array} \right. .$$
    
\subsection{Factor $3 x_{45} + 1$ - {Real 1}}

    We obtain 6 minimal primes, and a new kind of solution, ``{Real 1}''.
    $$ X_{N1} = \begin{pmatrix}
    	1 & x_{45} & x_{02} & x_{02} & x_{45} & x_{45} \\
    	& 1 & x_{02} & x_{02} & x_{45} & x_{45} \\
    	& & 1 & x_{23} & x_{02} & x_{02} \\
    	& & & 1 & x_{02} & x_{02} \\
    	& & & & 1 & x_{45} \\
    	& & & & & 1 \\
    \end{pmatrix}, \quad \begin{array}{l} x_{45} = -1/3, \\ x_{02} = 0, \\ x_{23} = - 1 \end{array} .$$

\subsection{Factor $5 x_{45} + 4$ - {Real 2}}

We obtain 6 minimal primes, and a new kind of solution, ``{Real 2}''.
$$ X_{N2} = \begin{pmatrix}
	1 & x_{45} & x_{02} & x_{02} & x_{04} & x_{04} \\
	& 1 & x_{02} & x_{02} & x_{04} & x_{04} \\
	& & 1 & x_{02} & x_{02} & x_{02} \\
	& & & 1 & x_{02} & x_{02} \\
	& & & & 1 & x_{45} \\
	& & & & & 1 \\
\end{pmatrix}, \quad \begin{array}{l} x_{45} = -4/5, \\ x_{04} = 1/10, \\ x_{02} = -1/5 \end{array} .$$

\subsection{Factor $25 x_{45} - 1$ - {Real 3}}

We obtain 24 minimal primes, and a new kind of solution, ``{Real 3}''.
$$ X_{N3} = \begin{pmatrix}
	1 & x_{45} & x_{45} & x_{03} & x_{45} & x_{05} \\
	& 1 & x_{12} & x_{03} & x_{12} & x_{45} \\
	& & 1 & x_{03} & x_{12} & x_{45} \\
	& & & 1 & x_{03} & x_{03} \\
	& & & & 1 & x_{45} \\
	& & & & & 1 \\
\end{pmatrix}, \quad \begin{array}{l} x_{45} = 1/25, \\ x_{05} = -23/25, \\ x_{12} = -11/25, \\ x_{03} = -1/5 \end{array} .$$

\subsection{Factor $x_{45}$ - {Real 4}}

We obtain 26 minimal primes, and a new kind of solution, ``{Real 4}''.
$$ X_{N4} = \begin{pmatrix}
	1 & x_{01} & x_{45} & x_{45} & x_{45} & x_{01} \\
	& 1 & x_{45} & x_{45} & x_{45} & x_{01} \\
	& & 1 & x_{01} & x_{01} & x_{45} \\
	& & & 1 & x_{01} & x_{45} \\
	& & & & 1 & x_{45} \\
	& & & & & 1 \\
\end{pmatrix}, \quad \begin{array}{l} x_{45} = 0, \\ x_{01} = -1/2 \end{array} .$$

\section{New real solutions for $n=6$} \label{sec:cartesian_coord_new}

In this appendix we give cartesian coordinates on $S^3$ for the new real solutions of six points, as described in Section \ref{sec:cartesian_coords_sphere}.

$$ W_{\text{R1}} = \frac{1}{3} \begin{pmatrix}
    0 & 0 &  -\sqrt{6} & \sqrt{6} & 0 & 0 \\
    0 & 0 & -\sqrt{2} & -\sqrt{2} & 2\sqrt{2} & 0 \\
    0 & 3 & -1 & -1 & -1 & 0 \\
    3 & 0 & 0 & 0 & 0 & -3
\end{pmatrix} .$$

	$$ W_{\text{R2}} = \frac{1}{10} \begin{pmatrix}
    -3\sqrt{10} & 3\sqrt{10} & 0 & 0 & 0 & 0 \\
    0 & 0 & -3\sqrt{10} & 3\sqrt{10} & 0 & 0 \\
    0 & 0 & 0 & 0 & 2\sqrt{15} & -2\sqrt{15} \\
    \sqrt{10} & \sqrt{10} & \sqrt{10} & \sqrt{10} & -2\sqrt{10} & -2\sqrt{10}
\end{pmatrix}.$$

	$$ W_{\text{R3}} = \frac{1}{5} \begin{pmatrix}
    0 & 0 & -3\sqrt{2} & 3\sqrt{2} & 0 & 0 \\
    0 & 0 & -\sqrt{6} & -\sqrt{6} & 2\sqrt{6} & 0 \\
    0 & 2\sqrt{6} & 0 & 0 & 0 & - 2\sqrt{6} \\
    5 & -1 & -1 & -1 & -1 & -1
\end{pmatrix} .$$

	\begin{equation} \label{eq:cartesian_new4}
		W_{\text{R4}} = \frac{1}{2} \begin{pmatrix}
        -\sqrt{3} & \sqrt{3} & 0 & 0 & 0 & 0 \\
        -1 & -1 & 2 & 0 & 0 & 0 \\
        0 & 0 & 0 & -\sqrt{3} & \sqrt{3} & 0 \\
        0 & 0 & 0 & -1 & -1 & 2
    \end{pmatrix}.
	\end{equation}

Based on the geometrical interpretation of Section \ref{sec: Geometric interpretation on S^3}, and inspired by the F\"oppl notation, we propose the following notation for these configurations:
$$ \begin{array}{ll}
    \text{({Real 1})} \quad 1:{4}_{S^2}:1, & \text{({Real 2})} \quad { 4_{\text{Eq.}} : 2_{S^2} } \\
    \text{({Real 3})} \quad 1:5_{S^2}, & \text{({Real 4})} \quad 3_{S^1} \times 3_{S^1}.
\end{array} $$

\section{Projected Hessian eigenvalues} \label{sec:eigenvalues_hessian}
    
    The eigenvalues of the projected Hessian are given in Tables \ref{tab:hess_eigenvalues_n4}, \ref{tab:hess_eigenvalues_n5} and \ref{tab:hess_eigenvalues_n6}.
    The associated code can be found in the \emph{SymPy} notebook \emph{``classifyHessian''}.
    For 1:5 and 3:3 we could not calculate the exact eigenvalues, and we report the eigenvalues calculated with floating point arithmetic. To prove that these configurations are saddle, we give feasible directions where the quadratic form of the Hessian is negative.
    To find a negative direction for 1:5, we move the first two points of the roots of unity, one upward and the other downward. The associated value of the quadratic form of the Hessian is:
    $$ v^T H_L v = \frac{ -263 \sqrt{5} }{625} + \frac{ 13 \sqrt{5} \sqrt{ 2 \sqrt{5} + 6 } }{625} + \frac{185}{625} \simeq -0.494 < 0 .$$
    To find a negative direction for 3:3, we move the three points of the upper hemisphere in the direction of a counter-clockwise rotation with respect to the z-axis. The associated value is: $$ v^T H_L v = 36 - 15 \sqrt{6} \simeq -0.742 < 0 .$$
    
    \begin{table}[H]
        \centering
        \caption{$n=4$. Eigenvalues of the Hessian of the Lagrangian, projected onto the tangent space of the product of spheres.}
        \label{tab:hess_eigenvalues_n4}
        \begin{tabular}{cccc}
        	$\mathbb{R}^d$ & Conf. & eig.: mult. & $|O(d)|$ \\
        	\midrule
        	$\mathbb{R}^2$ & Equator & 4:1, 3:2, 0:1 & 1 \\
        	\midrule
            $\mathbb{R}^3$ & Equator & 4:1, 3:3, 0:3, -1:1 & 3 \\
            & Tetrahedron & $\frac{3}{2}$:2, 3:3, 0:3 & \\
            \bottomrule
        \end{tabular}
    \end{table}
    
    \begin{table}[H]
        \centering
        \caption{$n=5$. Eigenvalues of the Hessian of the Lagrangian, projected onto the tangent space of the product of spheres.}
        \label{tab:hess_eigenvalues_n5}
        \begin{tabular}{cccc}
        	$\mathbb{R}^d$ & Conf. & eig.: mult. & $|O(d)|$ \\
        	\midrule
        	$\mathbb{R}^2$ & Equator & 6:2, 4:2, 0:1 & 1 \\
        	\midrule
            $\mathbb{R}^3$ & Equator & 6:2, 4:3, 0:3, -2:2 & 3 \\
            & 1:4 & $\frac{64}{15}$:1, $-\frac{4}{15}$:1, 2:2, 4:3, 0:3 & \\
            & 1:3:1 & $\frac{7}{2}$:2, $\frac{1}{2}$:2, 4:3, 0:3 & \\
            \midrule
            $\mathbb{R}^4$ & 1:3:1 & $\frac{7}{2}$:2, $\frac{1}{2}$:2, 4:4, 0:6, -1:1 & 6 \\
            & 4-simplex & 4:4, $\frac{8}{5}$:5, 0:6 & \\
            \bottomrule
        \end{tabular}
    \end{table}

    \begin{table}[H]
        \centering
        \caption{$n=6$. Eigenvalues of the Hessian of the Lagrangian, projected onto the tangent space of the product of spheres. Values with (*) are calculated with floating point arithmetic.}
        \label{tab:hess_eigenvalues_n6}
        \begin{tabular}{cccc}
        	$\mathbb{R}^d$ & Conf. & eig.: mult. & $|O(d)|$ \\
        	\midrule
        	$\mathbb{R}^2$ & Equator & 9:1, 8:2, 5:2, 0:1 & 1 \\
        	\midrule
            $\mathbb{R}^3$ & Equator & 9:1, 8:2, 5:3, 0:3, -4:1, -3:2 & 3 \\
            & 1:5 (*) & 5:3, 0:3, 2.5:2, -1.25:2, 6.25:2 & \\
            & 3:3 (*) & 5:3, 0:3, 3.55:2, 1.45:2, 5.80:1, -0.80:1 & \\
            & 1:4:1 & 5:3, 4:3, 1:3, 0:3 & \\
            \midrule
            $\mathbb{R}^4$ & 1:4:1 & 5:4, 4:3, 1:3, 0:6, -1:2 & 6 \\
            & {Real 1} & $\frac{3}{2}$:2, $\frac{10}{3}$:3, $\frac{1}{3}$:3, 5:4, 0:6 & \\
            & {Real 2} & $\frac{5}{3}$:1, $\frac{40}{9}$:1, $\frac{-5}{18}$:2, 5: 4, $\frac{25}{12}$:4, 0:6 & \\
            & {Real 3} & $\frac{13}{6}$:1, $\frac{-5}{24}$:1, $\frac{11}{6}$:2, $\frac{175}{48}$:2, $\frac{25}{48}$:2, 5:4, 0:6 & \\
            & {Real 4} & 5:4, 3:4, $\frac{1}{2}$:4, 0:6 & \\
            \midrule
            $\mathbb{R}^5$ & {Real 1} & $\frac{3}{2}$:2, $\frac{10}{3}$:3, $\frac{1}{3}$:3, 5:5, 0:10, -1:1 & 10 \\
             & {Real 4} & 5:5, 3:4, $\frac{1}{2}$:4, 0:10, -1:1 &  \\
            & 5-simplex & 5:5, $\frac{5}{3}$:9, 0:10 & \\
            \bottomrule
        \end{tabular}
    \end{table}


\printbibliography

\end{document}